\documentclass[10pt]{amsart}          

\usepackage{amsmath,amsthm}     
\usepackage{amssymb}            
\usepackage{euscript}           
\usepackage{enumerate,calc}
\usepackage[matrix,arrow,curve,frame]{xy}    
\xymatrixcolsep{1.9pc}                          
\xymatrixrowsep{1.9pc}
\usepackage{amscd}

\newdir{ >}{{}*!/-5pt/\dir{>}}                  

\numberwithin{equation}{section}

\newtheorem{thm}[equation]{Theorem}
\newtheorem{defn}[equation]{Definition}
\newtheorem{prop}[equation]{Proposition}

\newtheorem{cor}[equation]{Corollary}
\newtheorem{lemma}[equation]{Lemma}

\theoremstyle{definition}  
\newtheorem{example}[equation]{Example}

\newtheorem{remark}[equation]{Remark}

  {\end{list}}

\newcommand{\cat}{\EuScript}    
\newcommand{\cC}{{\cat C}}
\newcommand{\cD}{{\cat D}}

\newcommand{\cF}{{\cat F}}

\newcommand{\cN}{{\cat N}}

\newcommand{\cS}{{\cat S}}
\newcommand{\cV}{{\cat V}}
\newcommand{\cU}{{\cat U}}

\newcommand{\Gr}{{\cat Grpd}}

\newcommand{\Set}{{\cat Set}}
\newcommand{\cCat}{{\cat Cat}}
\newcommand{\sSet}{s{\cat Set}}

\newcommand{\field}[1]  {\mathbb #1} 

\newcommand{\Z}         {\field Z}

\DeclareMathOperator*{\holim}{holim}
\DeclareMathOperator*{\hocolim}{hocolim}

\DeclareMathOperator{\spec}{Spec}
\DeclareMathOperator{\Spec}{Spec}
\DeclareMathOperator{\Hom}{Hom}

\DeclareMathOperator{\ob}{ob}

\DeclareMathOperator{\Tot}{Tot}

\newcommand{\ra}{\rightarrow}                   
\newcommand{\lra}{\longrightarrow}              
\newcommand{\llra}[1]{\stackrel{#1}{\lra}}      

\newcommand{\we}{\llra{\sim}}                   

\newcommand{\fib}{\twoheadrightarrow}           

\newcommand{\inc}{\hookrightarrow}              
\newcommand{\dbra}{\rightrightarrows}           

\newcommand{\m}{\mathcal M}         
\newcommand{\Sh}{\operatorname{Sh}} 

\newcommand{\tuborg}{\left\{\begin{array}{ll}}
\newcommand{\sluttuborg}{\end{array}\right.}

\newcommand{\cO}{{\mathcal O}}

\newcommand{\Aff}{{\cat Af\!f}}

\newcommand{\Ring}{{\cat Ring}}

\newcommand{\fpi}{{\pi_{oid}}}

\DeclareMathOperator{\equ}{equalizer}

\begin{document}
\email{sjh@math.huji.ac.il+}

\title{Descent for quasi-coherent sheaves on stacks}

\author{Sharon Hollander}

\address{Department of Mathematics,
Hebrew University, Jerusalem, Israel}

\date{November 10, 2006}

\begin{abstract}
We give a homotopy theoretic characterization of sheaves
on a stack and, more generally, a presheaf of groupoids on an arbitary 
small site $\cC$.
We use this to prove homotopy invariance and generalized
descent statements for categories of sheaves and quasi-coherent sheaves.
As a corollary we obtain an alternate proof of a generalized change
of rings theorem of Hovey.
\end{abstract}

\maketitle

\section{Introduction}

The purpose of this paper is to continue the study of stacks
from the point of view of homotopy theory.  We generalize basic definitions
and constructions pertaining to (quasi-coherent) sheaves on stacks,
to an arbitrary small site $\cC$ and an arbitrary presheaf of groupoids
on $\cC$ (not just a stack). We use this point of view to give new
proofs of fundamental theorems in this setting.   

Classically, stacks are defined as those categories fibered in groupoids over $\cC$,
(or equivalently lax presheaves of groupoids on $\cC$) which satisfy descent 
\cite[Definition $4.1$]{DM}.

In \cite{H} we show that a category fibered in groupoids $F$ over $\cC$ 
is a stack if and only if the assignment 
satisfies the {\it homotopy sheaf condition}, that is, 
for each cover $\{U_i \ra X \in \cC \}$, the natural map
$$\xymatrix{ F(X) \ar[r]^-{\sim} &  \holim \Bigl( \prod F(U_i) \ar@2[r] &
\prod F(U_{ij}) \ar@3[r] & \prod F(U_{ijk}) \dots \Bigr)}$$ 
is an equivalence of categories, (where the homotopy limit here is taken in 
the category of small groupoids, denoted $\Gr$, which is a simplicial model 
category).

This characterization of stacks naturally leads to a model structure on 
categories fibered in groupoids over $\cC$, in which the fibrant objects 
are the stacks.  Similarly,  
one can consider the strict functors, or presheaves of groupoids on $\cC$,
denoted $P(\cC,\Gr)$.  Here too there is a {\it local} model structure, 
denoted $P(\cC,\Gr)_L$, in which the fibrant objects are those functors 
which are stacks or, equivalently, satisfy the homotopy sheaf condition.
Furthermore, there is a Quillen equivalence    
between these two model categories. (See \cite[Section $4$]{H}).

Since this paper will derive results about sheaves on stacks from the
ambient model category it makes no difference which of the 
Quillen equivalent model categories one chooses to work in.
For the sake of simplicity we will work in $P(\cC, \Gr)_L$.

Given a stack $\m$, on $\cC$, the category of sheaves on 
$\m$ \cite[Definition $4.10$]{DM} is defined as sheaves on the site $\cC/\m$.
The site $\cC/\m$ can be easily generalized to an arbitrary presheaf 
of groupoids $\m$ and site $\cC$ (see Section $2.1$).  Here
objects of $\cC/\m$ are morphisms $X \ra \m \in P(\cC,\Gr)$, with $X\in\cC$,
and the morphisms are triangles with a commuting homotopy. 
Covers in $\cC/\m$ are the collections of morphisms which forget 
to covers in $\cC$.  
Notice that the underlying category $\cC/\m$ is just the Grothendieck 
construction on the functor $\m:\cC^{op} \ra \Gr$. 
Also notice that if $\m$ is represented by an object $X \in \cC$ 
this is the usual topology on the over category $\cC/X$.

We prove that the category of sheaves on $\cC/\m$ is equivalent 
to the full subcategory of fibrations $F \fib \m$ in $P(\cC, \Gr)_L$ 
where the fibers $F(X) \fib \m(X)$ are discrete for each $X \in \cC$. 
In fact, this yields an embedding of sheaves on $\m$ as a full 
subcategory of the homotopy category $Ho(P(\cC,\Gr)_L/\m)$
(Corollary \ref{char-sheaves}).
 
Using this embedding we prove that a local weak equivalence
$\m \to \m'$ induces via the restriction functor an equivalence of categories
\[ \Sh(\cC/\m') \to \Sh(\cC/\m)\]
(Theorem \ref{eq-sh}).
This also holds for sheaves of abelian groups, simplicial sets, rings, modules.

We also present a definition of a quasi-coherent sheaf of modules over
a sheaf of rings $\cO$ in an arbitrary site, (see Definition \ref{qcdef}),
and show that the results of the previous paragraph holds for quasi-coherent sheaves (Corollary \ref{cor-qc}).

Classically, sheaves on an algebraic stack $\m$ are described via an atlas.
If $X \ra \m$ an atlas, sheaves on $\m$ can be written
as a sheaves on $X$, with an isomorphism of the two pullbacks to 
$X \times_\m X$ satisfying the cocycle condition.
 
We generalize this and prove the following {\it descent statement}:
given an $I$-diagram $\m_I$ in $P(\cC,\Gr)$  
there is an equivalence of categories
\[ \Sh(\cC/(\hocolim \m_i)) \to \holim \Sh(\cC/\m_i) \]
(where the homotopy limit is taken in $\cCat$ with the categorical model 
structure, see \cite{R}).
The same holds for sheaves taking values in any category with products 
(Proposition \ref{desc-sh}).  We prove also the 
analog for quasi-coherent sheaves of modules (Proposition \ref{desc-qcsh}).
We think of the diagram $\m_I$ as a generalized 
{\it presentation} of $\hocolim \m_I$.

Our descent statement generalizes the classical scenario since 
(\cite[Proposition A.$9$]{H2} given an atlas $X \ra \m$, 
the induced map below is a weak equivalence
$$\xymatrix{ \hocolim \Bigl( \cdots  X \times^h_\m X \times^h_m X  \ar@3[r] & 
X \times^h_\m X \ar@2[r] & X \Bigr) \ar[r]^\sim & \m,}$$
and so it follows that $Sh(\cC/\m)$ is the homotopy inverse limit 
of the categories
$$\xymatrix{ Sh(X) \ar@2[r] & Sh( X \times^h_\m X) \ar@3[r] & 
Sh( X \times^h_\m X \times^h_m X ) \dots }$$
which is a modern formulation of the classical statment written above
(see Section \ref{sectapp}).

A simple application of this descent statement (Proposition \ref{comodulesqc})
implies that the category of comodules over an Hopf algebroid $(A,\Gamma)$
is equivalent to the category of quasi-coherent sheaves on 
the presheaf of groupoids represented by the pair $(\spec A,\spec \Gamma)$ 
and so is also equivalent to quasicoherent sheaves on its stackification
$\m_{(A,\Gamma)}$ (which is its fibrant replacement in $P(\Aff_{flat},\Gr)_L$).

It follows that if $(A,\Gamma)$ and  $(B,\Gamma')$ are two weakly equivalent
Hopf algebroids then the categories of comodules on each 
are equivalent (Corollary \ref{eq-ha}).

The greater generality here is important for many reasons.
First we provide an elementary description of the category of sheaves 
on a stack which is independent of the choice of site and makes sense
for any stack $\m$ algebraic or not (compare with \cite[Chapter 12]{LM-B}). 
The descent statement shows that alternative descriptions of the category of 
sheaves on $\m$ can be obtained in many fashions, not just via an atlas,
and not only in geometric contexts. In particular, in the case of 
algebraic stacks
these description are not a by-product of the geometry but of category theory.
Enlarging one's frame of reference to include presheaves of groupoids which
are not stacks also enlarges the range of {\it presentations}
and so the possible alternative descriptions of one's category of sheaves.

The stack with the greatest relevance to 
stable homotopy theory is $\m_{FG}$ the moduli stack of formal groups
(see \cite{G,P,N}).
The Lazard ring provides an atlas $\spec L \ra \m_{FG}$.  But $L$ is 
not noetherian and the maps $\spec L \times_{\m_{FG}} \spec L \ra \spec L$
are not finitely presented.  It follows that $\m_{FG}$ is not an algebraic 
stack and much of the classical literature concerning sheaves on a stack 
does not apply.  

We believe that in the context of problems whose origin is homotopy theory,  
larger classes of {\it presentations} for stacks should naturally appear
and our descent statements will be of use.
 
Finally, as in \cite{H}, we believe that the proper context in which to 
understand stacks is a homotopy theoretic one.  Weak equivalences 
(or $2$-equivalences) of stacks are not homotopy equivalences.  Thus 
one can not work in a naive homotopy category of stacks and behave as 
if these equivalences were isomorphism and the $2$-category pullback were
a real pullback.  The only reasonable way to properly contextualize 
these equivalences and all the constructions one makes taking them 
into account is via a model category structure.  Abstract homotopy theory was
invented precisely to solve these types of problems.

\subsection{Relation to other work} 

Part of the results here are bringing those of \cite{Ho} 
into the homotopy theoretic framework of \cite{H}. 
In \cite{Ho}, Hovey defines quasi-coherent sheaves on a presheaf of groupoids
on $\Aff_{flat}$ and prove a generalized change of rings theorem.
It is a consequence of \ref{char-sheaves}, \ref{modules} that our definition of 
(quasi-coherent) sheaves agrees with Definitions \cite[$1.1,1.2$]{Ho}.
Our Proposition \ref{comodulesqc} then is exactly Theorem A in \cite{Ho}.
Proposition $5.7$ \cite{H} implies that the {\it internal equivalences} 
of (\cite[Definition $3.1$]{Ho}) agree with our local weak equivalences. 
It follows that Propositions \ref{eq-sh}, \ref{cor-qc} are exactly Theorems
B and C of \cite{Ho}, when the site $\cC=\Aff_{flat}$.  
Theorem D \cite{Ho} also follows directly from \cite[Proposition $5.7$]{H}.

\subsection{Acknowledgements} 
The project of understanding stacks from the point of view of homotopy theory was 
inspired by a course by M. Hopkins at M.I.T. and his ideas permeate this work. 
I would also like to thank G. Granja for many helpful comments.This 
research was partially supported by the
Center for Mathematical Analysis, Geometry,
and Dynamical Systems at the Instituto Superior T\'ecnico of the 
Technical
University of Lisbon and the Golda Meir Fellowship Trust at the
Hebrew University of Jerusalem.

\section{Background}

In this section we recall some results from the homotopy theory of categories, groupoids
and presheaves of groupoids from \cite{R,H} and in the course of this fix our notation and
conventions for the rest of the paper.
\subsection{Homotopy theory of categories}
\label{hothcats}

Recall that $\Gr$ has a cofibrantly generated simplicial model category structure in which:
\begin{itemize}  
\item weak equivalences are equivalences of categories, and 
\item fibrations are maps $p:G \ra H$ such that given $\alpha:b \we p(a) \in H$ 
there exists $\beta:c \ra a \in G$ with $p(\beta)=\alpha$.
\item the simplicial structure is inherited via the fundamental groupoid functor
$\fpi$. 
\item The generating trivial cofibration is $\ast \to \fpi\Delta^1$. The generating cofibrations are $\{\ast,\ast\} \to \fpi\Delta^1, B\Z \to \ast, \emptyset \to \ast$.
\end{itemize}

There is also a simplicial model category structure on $\cCat$
in which weak equivalences are equivalences of categories
and fibrations are the maps which have the right lifting 
property with respect to $* \ra \fpi \Delta^1$.
The simplicial structure on $\cCat$ is defined by setting
$\cC \otimes X= \cC \times\fpi X$ and $\cC^X= \cCat(\fpi X,\cC)$.  
For more details see \cite{R}. We will sometimes abuse notation and 
write $\Delta^1$ for $\fpi \Delta^1$.

It follows from \cite[18.1.2, 18.1.8, 18.5.3]{Hi} that we have the 
following explicit formulas for homotopy limits and colimits in $\cCat$.
The homotopy inverse limit of an $I$ diagram of categories $\cC_I$ is 
the equalizer
$$\prod_{ob(I)} \cC(i)^{\fpi(I/i)} \dbra
\prod_{i \ra j \in I} \cC(j)^{\fpi(I/i)}$$
which can also be described as the end of the functors
$\cC(-)$ and $\fpi(I/-)$, see \cite[18.3]{Hi}.
Similarly the homotopy colimit of the diagram is the coequalizer
$$\coprod_{i \ra j \in I} \cC(i) \times \fpi(j/I) \dbra
\coprod_{i \in I} \cC(i)\times \fpi(i/I)$$
or the coend $\cC(-) \otimes_I \fpi(-/I)$.

We can also compute the homotopy (co)limit by taking a cosimplicial 
(simplicial) replacement of our diagram and applying the $\Tot$ 
("geometric realization") functor.
 
As the simplicial structure on $\cCat$ derives from a $\Gr$ enrichment,
it follows from \cite[Theorems 2.9, 2.12]{H} that $\Tot$ is equivalent 
to $\Tot^2$.  The category $\Tot^2(\cC^\bullet)$ of a cosimplicial 
category $\cC^\bullet$, has
\begin{itemize}
\item objects pairs $(x, \alpha)$ where $x$ is an object of $\cC^0$
and $\alpha: d^0x \we d^1x$ is an isomorphism in $\cC^1$ satisfying
$s^0 \alpha = id_x$ and $d^2 \alpha \circ d^0 \alpha = d^1 \alpha$, and
\item morphisms $(x, \alpha) \ra (y, \beta)$ consist of 
$h:x \ra y \in \cC^0$ such that  $\beta \circ d^0 h= d^1 h \circ f$.
\end{itemize}
Using cosimplicial replacement one obtains from this formula a compact
description of an arbitrary homotopy limit.

Similarly a model for the homotopy colimit of a simplicial diagram of 
categories $\cC_\bullet$ is the coend in $\cCat$,
$\cC_\bullet \otimes_{\Delta} \fpi \Delta[-]$, 
which we also refer to as the \emph{geometric realization}, denoted $|\cC_\bullet|$. 
Here too we have a smaller model for $|\cC_\bullet|$ where the 
objects are the objects of $\cC_0$ and the morphisms are generated by
those in $\cC_0$ and the isomorphisms $f_y:d_0y \ra d_1y$ for each 
$y \in \cC_1$, subject to the relations: 
\begin{itemize}
\item $f_{s_0x}=id_x$
\item for $y \llra{g} y' \in \cC_1$, $d_1 g \circ f_y = f_{y'}\circ d_0g$, and
\item for $z \in \cC_2$, $f_{d_2z} \circ f_{d_0 z} = f_{d_1 z}$.
\end{itemize}

The formulas above also give descriptions of homotopy (co)limits in $\Gr$ (note that
the inclusion of $\Gr$ in $\cCat$ preserves limits and colimits).

\subsection{Sites and presheaves}

We will always assume that our sites $\cC$ are small and closed under finite products
and pullbacks. 
We write $P(\cC,\cD)$ for the category of presheaves on $\cC$ with values in a category 
$\cD$ and $P(\cC)=P(\cC,\Set)$.
We abuse notation and identify the objects in $\cC$ with 
the presheaves of sets (or discrete groupoids) they represent.

If $\{U_i \to X\}$ is a cover, we write $U=\coprod_i U_i$ for the coproduct of the
presheaves and $F(U)$ for $\Hom(U,F)= \prod_i F(U_i)$.
$U_\bullet$ is the \emph{nerve of the cover} which 
is the simplicial object obtained by taking iterated fiber products over $X$.
We will sometimes abuse notation and write a cover as $U \to X$.
$|U_\bullet|$ will denote the geometric realization of the simplicial object
in $P(\cC, \Gr)$. Recall that the geometric realization of a simplicial diagram $F_\bullet$
in $P(\cC,\Gr)$ is defined by $|F_\bullet|(Y) = |F_\bullet(Y)|$ (see \cite[Section 2.2]{H}).
In particular, $|U_\bullet|(Y)$ is the groupoid whose objects are $\coprod_i \Hom(Y,U_i)$
and whose isomorphisms are generated by the set $\coprod_{i,j} \Hom(Y, U_i \times_X U_j)$
satisfying the obvious relations (see the previous subsection).

We will consider two different model structures on the category of presheaves of groupoids.
$P(\cC,\Gr)$ will denote the \emph{levelwise model structure} where a map $F \to F'$ is a fibration (weak equivalence) if and only if $F(X) \to F'(X)$ is a fibration (weak equivalence) in $\Gr$. 
We will write $P(\cC,\Gr)_L$ for the \emph{local model structure} which is the localization
of $P(\cC,\Gr)$ with respect to the maps 
$$\{|U_\bullet| \to X \in P(\cC,\Gr)\}$$ 
where $\{U_i\to X\}$ is a cover in $\cC$. $F \in P(\cC,\Gr)_L$ is fibrant
iff $F(X) \to \holim_\Delta F(U_\bullet)$ is an equivalence of groupoids for all covers 
$\{U_i \to X\}$, i.e. iff $F$ is a stack (see \cite{H}). The stack condition is a direct
generalization of the sheaf condition since a presheaf of sets $F$ is a sheaf if and only
if $F(X) \to \lim_\Delta F(U_\bullet)$ is an isomorphism for all covers $\{U_i \to X\}$.

Note that, by definition of localization,
the cofibrations and trivial fibrations in $P(\cC,\Gr)_L$ are the same as those in $P(\cC,\Gr)$.
Unless otherwise noted, when we say a map of presheaves of groupoids $F\to G$ is a fibration or
weak equivalence we mean in the local model structure.

$P(\cC,\Gr)$ is enriched with tensor and cotensor over $\Gr$ in the obvious way and
therefore also over $\sSet$. Moreover, with this enrichment $P(\cC,\Gr)$ and $P(\cC,\Gr)_L$
are simplicial model categories (see \cite{H}).

Note that the geometric realization of a simplicial groupoid can be constructed by a finite sequence of pushouts along cofibrations and so $|F_\bullet|$ is cofibrant in $P(\cC,\Gr)$ so long as $F_0$, $F_1$ and $F_2$ are. In particular, if $\{U_i \to X\}$ is a cover, $|U_\bullet|$
is always cofibrant.

We say that a levelwise fibration $F \fib F'$ in $P(\cC,\Gr)$ has {\it discrete fibers}
if for each $X \in \cC$ the fiber of $F(X) \ra F'(X)$ over each object $a\in F'(X)$ 
is a discrete groupoid (i.e. a groupoid with only identity morphisms).

$\Hom_{P(\cC,\Gr)}(F,G)$ denotes the groupoid of maps between two presheaves of groupoids.
We write $h\Hom(A,B)$ for the homotopy function complex of maps between two objects
$A$ and $B$ in a model category. 

We will use repeatedly the following basic result \cite[Theorem 5.7]{H} characterizing the
weak equivalences in $P(\cC,\Gr)_L$ as those maps $F \llra{\phi} G \in P(\cC,\Gr)$ which satisfy the \emph{local lifting conditions}: 
\begin{enumerate}
\item
Given a commutative square
\[
\xymatrix{ \emptyset \ar[r] \ar[d] & F(X) \ar[d] \\ \star \ar[r] & G(X)}
\]
there exists a cover $U \ra X$ and lifts in the diagram as follows

\[
\xymatrix{\star \ar@/^3ex/@{-->}[rrr] \ar[d] & \emptyset \ar[l] \ar[r]
\ar[d] &  F(X) \ar[d] \ar[r] & F(U) \ar[d] \\
\Delta^1 \ar@/_3ex/@{-->}[rrr] & \star \ar[l] \ar[r] & G(X)\ar[r] &
G(U).} 
\]
\\
\item
For $A \ra B$, one of the generating cofibrations in $\Gr$ (see \cite[Section $2.1$]{H})
$\partial \Delta^1=\{\star,\star\} \ra \Delta^1$ or  $B\Z \ra \star,$ 
given a commutative square
\[
\xymatrix{ A \ar[r] \ar[d] & F(X) \ar[d] \\ B \ar[r] & G(X)}
\]
there exists a cover $U \ra X$ and a lift in the diagram as follows

\[
\xymatrix{A  \ar[r] \ar[d] &  F(X) \ar[d] \ar[r] & F(U) \ar[d] \\
B \ar[r] \ar@{-->}[rru] & G(X) \ar[r] & G(U).} 
\]
\end{enumerate}
Note that condition (1) means that $F \to G$ is locally essentially surjective while condition (2) 
says that $F \to G$ is locally full and faithful.

\section{Presheaves on a Stack}

In this section we associate a site to a presheaf of groupoids $\m$ 
and prove an equivalence of categories between presheaves of groupoids 
on this site and the full subcategory $(P(\cC,\Gr)/\m)_{df}$ of the over category 
consisting of levelwise fibrations with discrete fiber.

\subsection{Grothendieck Topology on $\m$}

The site we define is a simple generalization of the 
one first defined by Deligne and Mumford in \cite[Definition $4.10$]{DM}.

\begin{defn}
Let $\m$ be presheaf of groupoids on $\cC$ and let $\cC/\m$ denote the 
category whose 
\begin{itemize} 
\item objects are pairs $(X,f)$ where $X\in \cC$ and $X \llra{f} \m \in P(\cC,\Gr)$,
\item morphisms from $X \llra{f} \m$ to $X' \llra{g} \m$ are pairs 
$(h, \alpha)$ where $X \llra{h} X'$ and $\alpha$ is a homotopy 
$f \ra g \circ h$.  
\end{itemize}
\end{defn}

\begin{remark}
\begin{enumerate}[(a)]
\item Given maps $f,f':X \ra \m$, a homotopy $\alpha: f \ra f'$ 
determines an isomorphism in $\cC/\m$ between the objects $f$ and $f'$ and
 so a presheaf $F$ on $\cC/\m$ will satisfy $F(X,f)\cong F(X,f')$. 
\item The category $\cC/\m$ is just the Grothendieck construction 
on the functor $\m$, i.e. the category whose objects are pairs $(X,a)$ with $X \in \cC$
and $a \in \m(X)$ and morphisms defined in the obvious way.

Notice that the projection $\cC/\m \to \cC$ is $p\m$, the category fibered in groupoids associated to $\m$ \cite[Definition 3.11]{H}. 
\end{enumerate}
\end{remark}

\begin{lemma}
Let $(g, \alpha): (Y,g') \ra (X,f)$ and
$(h,\beta): (Z,h') \ra (X,f)$ be maps in $\cC/\m$.
The pullback of the maps $(g,\alpha)$ and $(h,\beta)$
in $\cC/\m$ is
$$(Y \times_X Z, f \circ (g \times h))$$ 
where  $g\times h$ denotes the canonical map $Y \times_X Z \llra{g\times h} X$. The
projection maps are $(p_Y, \alpha^{-1})$ and $(p_Z, \beta^{-1})$.
\end{lemma}
Using the previous lemma, the proof of the following proposition is an easy exercise.
\begin{prop}  
\label{top-on-stack}
Let $\cC$ be a site and $\m \in P(\cC,\Gr)$.
The collections of morphisms which forget to covers in $\cC$ form the 
basis for a Grothendieck topology on $\cC/\m$.
\end{prop}

\begin{remark}
The site of Proposition \ref{top-on-stack} generalizes the \'etale site \cite[4.10]{DM} of a Deligne-Mumford stack $\m$ which has
\begin{itemize} 
\item objects the schemes \'etale over $\m$, and 
\item morphisms triangles with a commuting homotopy, and 
\item covers those morphisms which forget to \'etale covers.
\end{itemize}
The Deligne-Mumford site is $\cC/\m$ when we take $\cC$ to be the 
category of schemes and \emph{\'etale maps} between them. If we 
take $\cC$ to be the category of schemes and all maps between them,
and $\m$ to be the sheaf represented by a scheme, the site of Proposition \ref{top-on-stack} is strictly bigger. It is called the big \'etale site of $\m$ \cite[II.3.3]{Ta}.

The sites that arise through the construction above are always over categories. This rules
out some examples such as the smooth-\'etale site of an algebraic stack \cite[Definition 12.1]{LM-B}.
\end{remark}

\begin{defn}
Let $T$ be a site. We say a collection of covers $S$ \emph{generates the topology} on $T$
when a presheaf $F$ on $T$ is a sheaf if and only if it satisfies
the sheaf condition when applied to a cover in $S$.
\end{defn}

\begin{prop} \label{covers} 
The collection of covers of the form
$$\{(U_i, f \circ u_i) \xrightarrow{(u_i,id)} (X, f)\}$$
generate the topology on $\cC/\m$.
\end{prop}

\begin{proof}
Consider an arbitrary cover $\{(U_i, f_i) \xrightarrow{(u_i,\alpha_i)} (X,f)\}$.
We may factor these maps as
$$(U_i, f_i) \xrightarrow{(id,\alpha_i)} (U_i, f \circ u_i) \xrightarrow{(u_i,id)} (X,f).$$
The first map is an isomorphism. If $F$ is a presheaf, the sheaf condition applied 
to the original cover requires that the top row in the following diagram be an equalizer
while the sheaf condition applied to $\{(U_i, f \circ u_i) \xrightarrow{(u_i,id)} (X,f)\}$
requires that the bottom row be an equalizer.
$$\xymatrix{ F(X,f) \ar[r] \ar[d]^= & \bigl( \prod F(U_i,f_i) \ar@2[r] 
\ar[d]^{F(id,\alpha_i)} & \prod F(U_i \times_X U_j, f \circ (u_i \times u_j)) 
\bigr) \ar[d]^= \\
F(X,f) \ar[r] &  \bigl( \prod F(U_i, f \circ u_i) \ar@2[r] 
&  \prod F(U_i \times_X U_j, f \circ (u_i \times u_j)) \bigr)}$$
Since the above diagram commutes these two conditions are equivalent.
\end{proof}

\subsection{Presheaves on $\m$}

We will now define an equivalence of categories between $P(\cC/\m)$ and 
the full subcategory of $P(\cC,\Gr)/\m$ consisting of levelwise fibrations
with discrete fibers, which we denote by $(P(\cC,\Gr)/\m)_{df}$.

\begin{defn}
Given $G \in P(\cC/\m)$, let $BG \in P(\cC, \Gr)$ be the presheaf so that
$BG(X)$ is the groupoid whose objects are pairs $(a, s)$ where $a \in \m(X)$
and $s \in G(X,a)$.  A morphism $(a,s) \ra (b,s')$ is a morphism
$a \llra{\alpha} b \in \m(X)$ such that $s= \alpha^*s'$.
\end{defn}
Alternatively, $BG(X)$ is the Grothendieck construction on the restriction of 
the functor $G$ to the subcategory $\m(X)$ of $\cC/\m$. The proof of the following
lemma is an easy exercise.
\begin{lemma}
The natural projection $BG \ra \m$ is a levelwise fibration with 
discrete fibers.
Moreover the fiber in $BG(X)$ over $a \in \m(X)$ is the set $G(X,a)$.
\end{lemma}
It is easy to check that $B$ defines a functor from $P(\cC/\m)$ to 
$(P(\cC,\Gr)/\m)_{df}$. 

A fibration $\pi\colon G \to H$ of groupoids with discrete fibers satisfies 
unique path lifting and so the assignment $a \mapsto \pi^{-1}(a)$ for $a \in \ob H$
defines a functor from $H^{op}$ to $\Set$.
Using this it is easy to see that the following definition makes sense.
\begin{defn}
The functor $\Gamma \colon (P(\cC,\Gr)/\m)_{df} \ra P(\cC/\m)$ is defined on objects 
by $\Gamma(F \llra{\pi} \m)(X,a) = \pi_X^{-1}(a)$ where $\pi_X^{-1}(a)$ denotes the fiber 
in $F(X)$ over $a \in \m(X)$. 
\end{defn}

\begin{prop} 
The pair $(B,\Gamma)$ is an adjoint equivalence of categories.
\end{prop}
\begin{proof}
There is a natural isomorphism $G \ra \Gamma BG$
which when evaluated at $(X,a)$ sends an element
$s \in G(X,a)$ to the element $(a,s)$ in $BG(X)$ lying over $a$.
Given $H \in (P(\cC,\Gr)/\m)_{df}$,
an element of $B \Gamma H(X)$ is a pair $(a,s)$ where $a \in \m(X)$
and $s$ is in the fiber of $H(X)$ over $a$. 
Sending $(a,s)$ to $s \in H(X)$ defines a natural isomorphism
$B\Gamma H \ra H$ over $\m$.
\end{proof}

\section{Sheaves}

In this section we identify sheaves on $\m$ with the category 
of fibrations $\cN \fib \m$ in $P(\cC,\Gr)_L$ with objectwise discrete fiber,
which we denote by $(P(\cC,\Gr)_L/\m)_{df}$.
This homotopy theoretic characterization of
the sheaves on a presheaf of groupoids allows us to prove invariance under 
weak equivalence.  We also extend these results to the categories of sheaves 
of rings and sheaves of quasi-coherent modules.

\begin{prop} The functors $(B,\Gamma)$ restrict to give an equivalence  
of categories between $Sh(\cC/\m)$ and $(P(\cC,\Gr)_L/\m)_{df}$.
\end{prop}
\begin{proof}
By Proposition \ref{covers} a presheaf $F$ on $\cC/\m$ is a sheaf if and only if 
$$F(X,a) \cong \equ \bigl( \prod F(U_i,a\circ u_i) \dbra 
\prod F(U_{ij},a \circ u_{ij}) \bigr)$$
for all covers $\{U_i \llra{u_i} X\}$ and $a \in \m(X)$.
Since 
$$F(X,a)=\Hom_{P(\cC,\Gr)/\m}(X \llra{a} \m, BF\ra \m)$$ 
the sheaf condition can be rewritten as
$$\Hom_{P(\cC,\Gr)/\m}(X,BF) \cong \lim_{\Delta} \Hom_{P(\cC,\Gr)/\m}(U_\bullet,BF).$$
As $BF \ra \m$ has discrete fibers, each groupoid of maps into it is 
discrete and therefore the inverse limit of $\Hom_{P(\cC,\Gr)/\m}(U_\bullet,BF)$
agrees with the homotopy inverse limit.

This shows that $F$ is a sheaf if and only if $BF \to \m$ is local with 
respect to the maps $|U_\bullet| \ra X \in P(\cC,\Gr)/\m$
where $U_\bullet$ is the nerve of a cover of $X$.
It follows from the following Proposition that 
this is equivalent to $BF \ra \m$ being a fibration.
\end{proof}

\begin{prop} \label{charac-fib}
A map $\cF \ra \m$ in $P(\cC,\Gr)_L$ is a fibration if and only if it
it is a levelwise fibration and satisfies \emph{descent for covers}, meaning for all covers 
$\{U_i \ra X\}$ in $\cC$, the following is a homotopy pullback square:
$$\xymatrix{ \cF(X) \ar[r] \ar[d] & \holim \cF(U_\bullet) \ar[d] \\
\m(X) \ar[r] & \holim \m(U_\bullet). }$$ 
\end{prop}
\begin{proof}
The outline of the proof follows the arguments in 
\cite[Lemma 7.2,7.3, Proposition 7.3]{DHI}.

Let $A \to B$ denote a generating cofibration
$\emptyset \to \ast$, $B\Z \to \ast$, $\{\ast,\ast\} \to \Delta^1$ in $\Gr$.
Let $J$ be the set of morphisms in $P(\cC,\Gr)$ consisting of 
\[ Z  \to Z \times \Delta^1, Z \in \cC\]
and
\[ |U_\bullet| \times B \coprod_{|U_\bullet| \times A} \tilde X \times A \to \tilde X \times B \]
where $\{U_i \to X\}$ is a cover and $|U_\bullet| \to \tilde X \to X$ is the factorization
of the natural map into a cofibration followed by a trivial fibration in $P(\cC,\Gr)_L$.

We claim that a map $\cF \to \m$ has the right lifting property with respect to the 
morphisms in $J$ iff it satisfies descent for covers and is a levelwise fibration. 
First note that a map $\cF \to \m$ is a levelwise fibration 
 iff it has the right lifting property with respect to the maps $\{Z  \to Z \times \Delta^1\}$.
Next observe that for a levelwise fibration $\cF \to \m$, the canonical map 
\begin{equation}
\label{mapg}
\Hom(\tilde X,\cF) \lra \Hom(|U_\bullet|, \cF)\times_{\Hom(|U_\bullet|, \m)} \Hom(\tilde X,\m)
\end{equation}
is a fibration (because $|U_\bullet| \to \tilde X$ is a cofibration in the levelwise model
structure $P(\cC,\Gr)$). A levelwise fibration $\cF \to \m$ satisfies the right lifting property with respect to $J$ if and only if the map \eqref{mapg} is a trivial fibration, and
therefore if and only if the following square is homotopy cartesian
\[ \xymatrix{ \Hom(\tilde X,\cF) \ar[r] \ar[d] &  \Hom(|U_\bullet|, \cF) \ar[d] \\ 
\Hom(\tilde X,\m) \ar[r] &  \Hom(|U_\bullet|, \m).  } \]
By definition, $\tilde X \ra X$ is a trivial fibration and therefore a 
levelwise weak equivalence. Since $X$ and $\tilde X$ are cofibrant and
 all objects are levelwise fibrant, 
\[ \Hom(X,\cF) \to \Hom(\tilde X, \cF), \qquad \Hom(X,\m) \to \Hom(\tilde X, \m) \]
are weak equivalences. This completes the proof of the claim.

It now suffices to show that $J$ provides a set of generating trivial 
cofibrations for $P(\cC,\Gr)_L$.
By \cite[Lemma 7.3]{DHI} it is enough to show that if $\cF \to \m$ is
a weak equivalence, a levelwise fibration, and satisfies descent for covers then
it is in fact a levelwise trivial fibration. We'll check the right lifting property
of $\cF(X) \to \m(X)$ with respect to the generating cofibrations in $\Gr$ for 
every $X \in \cC$.

Given a diagram
\[ \xymatrix{
B\Z \ar[d] \ar[r]^\alpha & \cF(X)\ar[d] \ar@{-->}[r] & \cF(U) \ar[d] \\ 
\ast \ar[r] & \m(X)   \ar@{-->}[r] & \m(U)  } \]
there exists a cover $U \to X$ such that the isomorphism $\alpha$ becomes trivial in $\cF(U)$.
As $\holim \cF(U_\bullet) \to \cF(U)$ is faithful, it follows that $\alpha$ is also trivial in $\holim \cF(U_\bullet)$. As 
$\cF \to \m$ satisfies descent for covers, $\alpha$ must be trivial to begin with.
This shows that $\cF(X) \to \m(X)$ is faithful for all $X$.

Given a commutative square
\[ \xymatrix{
\{\ast,\ast\} \ar[d] \ar[r] & \cF(X)\ar[d] \ar@{-->}[r] & \cF(U) \ar[d]  
\ar@2[r] & \cF(U\times_X U) \ar[d] \\ 
\Delta^1 \ar[r] \ar@{-->}[rru]^\alpha & \m(X)   \ar@{-->}[r] & \m(U) \ar@2[r] & 
\m(U\times_X U)  } \]
The local lifting conditions provide us with a lift $\alpha$ and the two images 
of $\alpha$ in $\cF(U\times_X U)$ lie over the same morphism in $\m(U\times_X U)$.
Since $\cF \to \m$ is levelwise faithful it follows that $\alpha$ is equalized by
the two maps. Thus $\alpha$ gives rise to a morphism in $\holim \cF(U_\bullet)$
and hence in $\cF(X)$. Thus $\cF(X) \to \m(X)$ is also full.

Given $a \in \m(X)$, the local lifting conditions provide us with a cover $U \to X$
such that $a$ lifts to an element in $\cF(U)$. As the two images in $\m(U\times_X U)$
are isomorphic, levelwise fullness implies that they are isomorphic in $\cF(U\times_X U)$
and this isomorphism satisfies descent by faithfulness of the map
$\cF(U\times_X U\times_X U) \to \m(U\times_X U\times_X U)$. This provides
us with an element in $\holim \cF(U_\bullet)$ and as $\cF \to \m$ satisfies 
descent for covers, this element lifts to $\cF(X)$ up to isomorphism. This proves
essentially surjectivity of $\cF(X) \to \m(X)$ and completes the proof.
\end{proof}

\begin{cor}
Given levelwise fibrations $\cF \we \cF'\fib \m$ with the first 
map a levelwise weak equivalence, then $\cF' \ra \m$ is a fibration
if and only if $\cF \ra \m$ is a fibration.
\end{cor}

\begin{prop}
If $G$ is a sheaf on $\m$ and $H \in P(\cC,\Gr)_L/\m$ then 
$$\Hom_{P(\cC,\Gr)/\m}(H,BG)\cong h\Hom_{P(\cC,\Gr)_L/\m}(H,BG) = [H,BG]_{P(\cC,\Gr)_L/\m}.$$
\end{prop}
\begin{proof}
Let $K$ denote the cofibrant replacement of $H$ in $P(\cC,\Gr)$.
We need to show that $\Hom_{P(\cC,\Gr)/\m}(K,BG) \cong \Hom_{P(\cC,\Gr)/\m}(H,BG)$.
Given $X \in \cC$, the map $K(X) \to H(X)$ is a trivial fibration, so 
writing $K(X)_a$ for the fiber over $a \in \ob H(X)$, we have 
a pushout square in $\Gr$
$$\xymatrix{\coprod_{a \in ob(H(X))} K(X)_a \ar[r] \ar[d] & K(X) \ar[d] \\ 
ob(H(X)) \ar[r] & H(X).}$$
Since the fibers of $G(X) \to \m(X)$ are discrete, it follows that there
is a unique extension
$$\xymatrix{K(X) \ar[r] \ar@{->>}[d]_\sim & G(X) \ar@{->>}[d] \\ 
H(X) \ar@{-->}[ur]^{!} \ar@{->>}[r] & \m(X)}$$
so every map from $K$ factors uniquely through $H$. This proves the first
equivalence in the statement.

On the other hand, since $BG \to \m$ has discrete fibers, two maps
$K \to BG \in P(\cC,\Gr)/\m$ are homotopic if and only if they are equal,
which completes the proof.
\end{proof}
If $F,F'$ are presheaves on $\m$, $\Hom_{P(\cC,\Gr)/\m}(BF,BF')$ is a discrete 
groupoid so we have the following corollary.
\begin{cor}
\label{char-sheaves}
The composition 
$$Sh(\cC/\m) \llra{B} P(\cC,\Gr)_L/\m \ra Ho(P(\cC,\Gr)_L/\m)$$ 
induces an equivalence of $Sh(\cC/\m)$ with the full subcategory of 
the homotopy category which consists of fibrant objects 
with levelwise discrete fiber. 
\end{cor}

We can now prove the main result of this section which states the invariance of the categories of sheaves under a local equivalence of presheaves of groupoids.
\begin{thm}\label{eq-sh}
A weak equivalence $\m' \llra{p} \m$ in $P(\cC,\Gr)_L$ 
induces a Quillen equivalence 
$$L:P(\cC,\Gr)_L/\m' \leftrightarrow P(\cC,\Gr)_L/\m:R.$$
The induced equivalence of homotopy categories 
yields an equivalence of categories
\begin{equation} \label{equiv-sh} 
p_*:Sh(\cC/\m') \leftrightarrow Sh(\cC/\m):p^*
\end{equation}
where the right adjoint $p^*$ is composition with 
$\cC/\m' \llra{p} \cC/\m$ and the left adjoint $p_*$ is the 
left Kan extension along $p$ followed by sheafification. 
\end{thm}

\begin{proof}
For first statement it suffices to observe that $P(\cC,\Gr)_L$ is
right proper \cite[Corollary 5.8]{H}. $L$ is composition with $p$ and $R$ is pullback
by $p$.

The derived functor $\underline{R}$ is just the pullback when applied to 
fibrant objects, and the pullback of a fibration with discrete fibers 
is also one.
Furthermore the sections of $F \times_\m \m'(X)$ over $a\in \m'(X)$
are exactly the sections of $F(X)$ over $p(a) \in \m(X)$, and so 
$\underline{R}$ agrees with $p^*$ when applied to the image of a sheaf on $\m$.
It follows that $p^*$ is full and faithful and it remains to show that
 $p^*$ is essentially surjective.

The functor assigning to $\cF' \llra{f} \m'$ the second map in the 
factorization of $p \circ f$
$$\cF' \we \cF \fib \m$$
as trivial cofibration followed by a fibration
provides a model for the derived functor $\underline{L}$.

Given a fibration $\cF' \llra{f} \m'$ we have the following 
commutative diagram  
$$\xymatrix{\cF' \ar[r] \ar@{->>}[dr]^f & 
\m' \times_\m \cF \ar[r]^\sim \ar@{->>}[d] & 
\cF \ar@{->>}[d]^{\underline{L}f} \\
& \m' \ar[r]^\sim & \m. }$$
By two out of three the map $\cF' \ra \m' \times_\m \cF$ is a 
weak equivalence between fibrant objects and is therefore a levelwise
weak equivalence.

Given a sheaf $F'$ on $\m'$, applying $\underline{L}$ to $BF' \llra{f} \m'$ 
yields a fibration $\cF \ra \m$.  The fiber of $\cF \ra \m$ over $X \ra \m$ is
$$\Hom_{P(\cC,\Gr)/\m}(X,\cF) =  h\Hom_{P(\cC,\Gr)/\m}(X,\cF) \simeq $$
$$ \simeq h\Hom_{P(\cC,\Gr)/\m'}(\m' \times_\m X,\m' \times_\m\cF)$$
where the equivalence arises from $R$ being part of a Quillen equivalence.
Since $\m' \times_\m\cF \ra \m'$ is a fibration with homotopically 
discrete fibers the homotopy function complex of maps from 
any object in $P(\cC,\Gr)/\m'$ into it is homotopically discrete.
It follows that $\cF \ra \m$ also has homotopically discrete fibers.

Given $\cF \fib \m$ a fibration with levelwise homotopically discrete fibers
a variation on the construction in the last proof can be used to 
construct a factorization $\cF \ra \cF' \ra \m$, where $\cF \ra \m$ is 
a levelwise fibration with discrete fibers and $\cF \ra \cF'$ is a 
levelwise trivial fibration. It follows that $\cF' \ra \m$ is a fibration
 $P(\cC,\Gr)_L$ from the characterization of fibrations in Proposition \ref{charac-fib}.
We conclude that $\cF \ra \m$ is isomorphic to a sheaf in $Ho(P(\cC,\Gr)/\m)$.
The pullback of $\cF \ra \m$ is weakly equivalent to $BF'$
and so $p^*:Sh(\m) \ra Sh(\m')$ is essentially surjective.

Finally the description given for $p_*$ follows as it is the left
 adjoint of $p^*$.
\end{proof}

\begin{remark}
Since sheaves of abelian groups are just abelian group objects in 
the category of sheaves Theorem \ref{eq-sh} also yields an equivalence
of sheaves of abelian groups.  Similarly we obtain equivalences of sheaves
of rings, simplicial sets, and have the following equivalence for sheaves of
modules (see \cite[p.95]{MM}).  
\end{remark}

\begin{cor} \label{cor-qc}
Let $\m' \llra{p} \m$ be a weak equivalence in $P(\cC,\Gr)_L$.
Let $\cO$ be a sheaf of rings on $\m$, and
$\cO'= p^*\cO$, then $p^*$ induces an equivalence of 
categories $(\cO-mod) \llra{p^*} (\cO'-mod)$.
\end{cor}

I learned the following definition from M. Hopkins\footnote{
It follows from faithfully flat descent (see Section \ref{quasi-coh-flat}) that this definition generalizes \cite[Definition 13.2.2]{LM-B} for the \'etale site of a Deligne-Mumford stack. Roughly speaking the difference between the definitions is that \cite{LM-B} requires a sheaf to be 
globally presentable, in the sense that for each $X \ra \m$ the sheaf has what we have called 
a presentation, while we only require this to hold locally.}.
\begin{defn} 
\label{qcdef}
Let $\cO$ be a sheaf of rings on $\m$.
A \emph{quasi-coherent sheaf of modules relative to $\cO$} is an
$\cO$ module $\cF$ which is locally presentable.  
This means that for every $X \llra{a} \m$ 
there exists a cover $\{U_i \llra{u_i} X\}$ in $\cC$ and 
exact sequences of $(a \circ u_i)^*\cO$-modules
$$\oplus_I (a \circ u_i)^*\cO \ra \oplus_J (a \circ u_i)^*\cO 
\ra (a \circ u_i)^*\cF \ra 0.$$
\end{defn}

The category $\cO-mod_{qc}$ of quasi-coherent modules is by definition the 
full subcategory of $\cO-mod$ whose objects are quasi-coherent sheaves.
This is not necessarily an abelian category. Even if it is an abelian category, the 
inclusion of $\cO-mod_{qc}$ in $\cO-mod$ is not necessarily exact.
See Section \ref{quasi-coh-flat} for a discussion of this in the case of 
affine schemes in the flat topology.

\begin{cor}
\label{equivqc}
Let $\cO$ be a sheaf of rings on $\m$.
A weak equivalence $\m' \llra{p} \m$ in $P(\cC,\Gr)_L$ induces  
an equivalence of categories between quasi-coherent $\cO$ modules
and quasi-coherent $p^*\cO$ modules.
\end{cor}

\begin{proof}
It is obvious that $p^*$ applied to a quasi-coherent $\cO$ module 
is a quasi-coherent $p^*\cO$ module.  

Conversely, let $M$ be an $\cO$ module such that $p^*M$ is a
quasi-coherent $p^*\cO$ module. 
Given $X \llra{a} \m$, it follows from the local lifting conditions that 
we can find a cover of the form $(U,p(b)) \xrightarrow{(u,\alpha)} (X,a)$
such that $p(b)^*M$ is presentable.
The isomorphism $p(b) \llra{\alpha} a\circ u$ induces a natural isomorphism
between the functors $(p(b))^*$ and $(a \circ u)^*$ from sheaves of rings
on $\m$ to sheaves of rings on $U$.
It follows that $(p(b))^*\cO \cong (a \circ u)^*\cO$ and 
that there is an equivalence of categories between quasi-coherent modules over
$(p(b))^*\cO$ and $(a \circ u)^*\cO$.  Since 
$(p(b))^*M = b^*p^*M$ is quasi-coherent $(a \circ u)^*M$ is also
quasi-coherent and hence so is $M$.
\end{proof}

We note that while $p^*$ is always exact as a functor between categories
of $\cO$ modules, it will not in general be exact when restricted 
to $\cO-mod_{qc}$.   

\subsection{Application: sheaf cohomology spectral sequence}
\label{sectapp}
If $\cF$ is a sheaf of abelian groups on $\m$ we can regard $\cF$ 
as an abelian group object in $P(\cC,\Gr)/\m$.
The global sections $\Gamma(\cF)$ are isomorphic
to the discrete simplicial abelian group $\Hom_{P(\cC,\Gr)/\m}(\m,\cF)$. 
Let $\{\cU_i \ra \m\}$ be a collection of maps such that the induced map
$|\cU_\bullet| \ra \m$ is a weak equivalence.
Then we have weak equivalences of simplicial abelian groups

$$\Hom_{\m}(\m,\cF) \cong \Hom_{\m}(|\cU_\bullet|, \cF) 
\cong \lim \Hom_{\m}(\cU_\bullet, \cF)$$
$$ \cong \lim \Hom_{\cU_i}(\cU_i, \cF \times_\m \cU_i) \cong 
\lim \Gamma(\cF \times_\m \cU_i).$$
The Grothendieck spectral sequence for composition of functors
in this case yields a spectral sequence with $E_2$-term
$$\check{H}^i(R^j\Gamma(\cF \times_\m \cU_i))\Rightarrow R^{i+j}\Gamma(\cF).$$
This a generalization of the usual \v{C}ech cohomology spectral sequence for a cover
which holds by the usual proof (see \cite[Theorem I.3.4.4]{Ta}).

\section{Descent for Sheaves on $\m$}

In this section we use the homotopy theory of categories recalled
in Section \ref{hothcats} and a notion of homotopy decomposition of a site to prove descent
statements for categories of sheaves on $\m \in P(\cC,\Gr)_L$.
A very special case of these statements yields a 
characterization of sheaves (of quasi-coherent modules)
on the stack associated to a groupoid object in $\cC$.

In the case of affine schemes in the flat topology $\Aff_{flat}$
this says that quasi-coherent sheaves on the stack associated to a 
Hopf algebroid $(A, \Gamma)$ is equivalent to the category of 
$(A,\Gamma)$-comodules.  Combining this result with Theorem \ref{eq-sh} 
gives an alternate proof of a generalized change of rings theorem due
to Mark Hovey \cite{Ho}.

\subsection{Descent}

In order to phrase our descent statement for categories of sheaves we need
the following definition.
\begin{defn}  
Let $T$ be a site.  A \emph {homotopy decomposition} of $T$ is
an $I$ diagram of sites $T_I$ and an equivalence $\hocolim T_i \we T$ 
such that 
\begin{enumerate}
\item the induced maps $T_i \ra T$ are maps of sites, 
\item the images in $T$ of all the covers in $T_i$ generate the topology.
\end{enumerate}
\end{defn}

\begin{prop} Let $\cD$ be any category with products.
A homotopy decomposition $\hocolim T_i \we T$ induces equivalences 
of categories
$$P(T,\cD) \we \holim P(T_i,\cD)$$
$$Sh(T,\cD) \we \holim Sh(T_i,\cD)$$
\end{prop}

\begin{proof}
Even though $\cD$ is not necessarily a small category, 
the functor categories are well defined and 
the duality of our presentations of $\hocolim$ and $\holim$ in
the previous subsection imply 
$$(\cD^{op})^{\hocolim T_i}= \holim (\cD^{op})^{T_i}$$
for any $I$ diagram of categories $T_I$ and any category $\cD$.

The proof for sheaves is an easy application of the following lemma
applied to the cosimplicial replacements of our diagrams.
\end{proof}

\begin{lemma} \label{subcat}
Let $\cD^\bullet \ra \cC^\bullet$ be a map of cosimplicial categories
such that each $\cD^i \ra \cC^i$ is a full subcategory, then 
$$\Tot(\cD^\bullet) = \cD^0 \times_{\cC^0} \Tot(\cC^\bullet).$$
In particular, if $\cC \we \Tot(\cC^\bullet)$ and $\cD \inc \cC$ is 
the full subcategory consisting of objects whose images in $\cC^0$ lie in 
the subcategory $\cD^0$ then $\cD \we \Tot(\cD^\bullet)$.
\end{lemma}

\begin{prop}
Let $\cU_I$ be an $I$ diagram in $P(\cC,\Gr)$.
There is a canonical homotopy decomposition 
$$\hocolim (\cC/\cU_i) \we \cC/(\hocolim \cU_i).$$
\end{prop}
\begin{proof}
Recall that $\cC/\cU_i$ is the Grothendieck construction on the functor
$\cU_i:\cC^{op} \ra \Gr$, or the coend $\cC/(-) \otimes_\cC \cU_i$.
An $I$ diagram $\cU_I$ in $P(\cC,\Gr)$ is a functor 
$\cC^{op} \times I \ra \Gr$.  Since coends commute we have 
$$\hocolim(\cC/U_i) = (\cC/(-) \otimes_\cC \cU_I) \otimes_I \fpi(-/I) \cong$$
$$\cong \cC/(-) \otimes_\cC (\cU_{I} \otimes_I \fpi(-/I)) 
= \cC/(\hocolim \cU_i).$$
Using the presentation of the homotopy colimit obtained by simplicial 
replacement of the diagram, we see that all maps $X \ra (\hocolim \cU_i)$ 
factor through some $\cU_i$.  It follows from Proposition \ref{covers}
that the equivalence is a homotopy decomposition.
\end{proof}
The main descent statement of this section is the following corollary of
the previous proposition.
\begin{prop} \label{desc-sh}
Let $\cD$ be any category with products, $\cU_I$ be an
$I$ diagram in $P(\cC,\Gr)_L$ and $\hocolim \cU_i \we \m$ a weak equivalence.
There is an equivalence of categories of presheaves
$$P(\cC/\m, \cD) \we \holim P(\cC/U_i, \cD)$$
and sheaves
$$Sh(\cC/\m, \cD) \we \holim Sh(\cC/U_i , \cD).$$
\end{prop}
The previous result yields the following more explicit description
of the category of sheaves on $\m$.
\begin{cor} \label{holim-sh}
Let $\cU_I$ be an $I$ diagram in $P(\cC,\Gr)_L$ 
and $\hocolim \cU_i \we \m$ a weak equivalence.
The category of $Sh(\m)$ is equivalent to the category whose 
\begin{itemize}
\item objects are collections $\{F_i,\alpha_f\}$ where 
\begin{enumerate}[(i)] 
\item $F_i$ is a sheaf on $\cU_i$,
\item $\alpha_f\colon f^*F_j \to  F_i$ is an isomorphism,
\end{enumerate}
satisfying $\alpha_{id_i}=id_{F_i}$ and $\alpha_{g \circ f} = \alpha_f \circ f^*(\alpha_g)$
for each $i \in \ob I$ and $i\llra{f} j \llra{g} k \in I$,
\item morphisms $\{F_i,\alpha_f\} \ra \{F_i',\beta_f\}$ are maps 
$\phi_i:F_i \ra F'_i \in Sh(\cU_i)$
such that $\phi_i \circ \alpha_f = \beta_f \circ f^*\phi_j $.
\end{itemize}
\end{cor}

\subsection{Descent for Quasi-coherent Sheaves}

Next we prove a version of these results for quasi-coherent sheaves.

\begin{defn}  A ringed space in $P(\cC,\Gr)$,
is a pair $(\cU, \cO_{\cU})$, where $\cU \in P(\cC,\Gr)$, and 
$\cO_{\cU}$ is a sheaf of rings on $\cU$.
A morphism of ringed spaces in $P(\cC,\Gr)$,
$(\cU, \cO_{\cU}) \ra (\cV,\cO_{\cV})$ consists of a morphism
$f:\cU \ra \cV \in P(\cC,\Gr)$ and an isomorphism $\cO_{\cU} \we f^* \cO_{\cV}$
of sheaves of rings on $\cU$.
\end{defn}

\begin{example}
If $\cC = \Aff$ with any reasonable topology and $\m \in P(\cC,\Gr)$, the assignment $\cO_{\m}(\Spec R \to \m) = R$ yields a ringed space.
\end{example}

An $I$ diagram of ringed spaces $(\cU_I,\cO_I)$ consists of 
an $I$ diagram $\cU_I$ in $P(\cC,\Gr)$ together with
sheaves of rings $\cO_i$ on $\cU_i$ and
for each $i\llra{\phi} j$ isomorphisms
$\phi^*\cO_j \we \cO_i$ of sheaves of rings on $\cU_i$ satisfying descent (i.e.
the conditions in Corollary \ref{holim-sh}). 
Such a diagram gives rise to an $I^{op}$ diagram of categories
$$i \mapsto \cO_i-mod$$
which assigns to a morphism $i\llra{\phi} j \in I$ the composite functor 
$$\cO_i-mod \quad \lra \quad \phi^*\cO_i-mod  \quad \lra \quad \cO_j-mod.$$

A diagram of ringed spaces $(\cU_I, \cO_I)$ gives rise to an element 
$[\cO_I] \in \holim(Sh(\cU_i, \Ring))$.  Using Corollary \ref{holim-sh}
one can see that $[\cO_I]$ is a ring object in $\holim(Sh(\cU_i, \Set))$ and
it is straightforward to check that the category of modules over $[\cO_I]$ is 
equivalent to the homotopy inverse limit of the $I^{op}$ diagram of categories 
$i \mapsto \cO_i-mod$. As a consequence we have the following result.

\begin{prop}
\label{desc-qcsh}
Let $(\cO_I, \cU_I)$ be an $I$ diagram of ringed spaces in $P(\cC,\Gr)$.
Let $\cO$ be a sheaf of rings on $\hocolim \cU_i$ which is 
isomorphic to $[\cO_I] \in \holim Sh(\cU_i,\Ring)$.
Then
$$\cO-mod \we \holim(\cO_i -mod)$$
and this equivalence restricts to an equivalence for quasi-coherent modules 
$$\cO-mod_{qc} \we \holim(\cO_i-mod_{qc}).$$
\end{prop}
\begin{proof}
Since $Sh(\hocolim \cU_i, \Set) \we \holim Sh(\cU_i, \Set)$
the categories of modules over the ring objects $\cO$ and $[\cO_I]$
are equivalent, and the category of modules over $[\cO_I]$ is equivalent 
to $\holim(\cO_i -mod)$. 
The proof for quasi-coherent modules follows by an application of
Lemma \ref{subcat}.
\end{proof}

\subsection{Descent for $\m_{(X_0,X_1)}$}

A groupoid object $(X_0,X_1) \in \cC$ determines a simplicial diagram 
in $\cC$:
$$\xymatrix{ \dots  X_1 \times_{X_0} X_1 \times_{X_0} X_1 \ar@3{->}[r] &
X_1 \times_{X_0} X_1 \ar@3[r]^{\ \ \ \ p_1,\mu}_{\ \ \  p_2} & X_1 \ar@2[r]^d_r & X_0.}$$
and therefore a simplicial diagram in $P(\cC,\Gr)$ 
which we denote by $(X_0,X_1)_\bullet$.
By definition, the presheaf of groupoids represented by $(X_0,X_1)$
is the geometric realization of this simplicial diagram and so
there is a weak equivalence
$$|(X_0,X_1)_\bullet| \we \m_{(X_0,X_1)}.$$
Using the model for the homotopy limit of a cosimplicial diagram given by $\Tot^2$
we see that an instance of Proposition \ref{desc-sh} is the following result.
\begin{cor} 
\label{descentsheaves}
The category of sheaves on a $\m_{(X_0,X_1)}$ 
is equivalent to the category with 
\begin{enumerate}
\item objects $(F,\alpha)$ with $F$ a sheaf on $X_0$ and $\alpha:d^*F \ra r^*F$ 
an isomorphism satisfying $i^*(\alpha)=id_F$ and $p_2^*(\alpha)\circ p_1^*(\alpha)=\mu^*(\alpha)$,
\item morphisms the maps of sheaves $\phi\colon F \to F'$  on $X_0$ satisfying
$r^*(\phi) \circ \alpha = \alpha' \circ d^*(\phi)$.
\end{enumerate}
\end{cor}
Similarly, given a sheaf of rings $\cO$ on $\m$, let $\cO_0$ be the pullback of $\cO$ to
$X_0$. The category of quasi-coherent $\cO$-modules is equivalent to the category with
objects $(F,\alpha)$ with $F$ a quasi-coherent $\cO_0$-module and  
$\alpha:d^*F \ra r^*F$ an isomorphism of $d^*\cO_0$-modules (where $r^*F$ is regarded
as a $d^*\cO_0$-module via the canonical isomorphism $d^*\cO_0 \simeq r^*\cO_0$) satisfying
the relations above.

\subsection{Quasi-Coherent Sheaves on a Hopf Algebroid}
\label{quasi-coh-flat}

In this section $\cC$ is the category 
affine schemes (and all morphisms between them) with the flat topology.
A groupoid object $(\spec A,\Spec \Gamma)$ in $\Aff_{flat}$
is called a Hopf algebroid.

Given $\m \in P(\Aff_{flat},\Gr)$ there is a natural choice 
of ``structure sheaf'' of rings $\cO_\m$ defined by 
$$\cO_\m(\spec R, a)=R.$$  
For the rest of this section quasi-coherent sheaves will always refer to 
quasi-coherent modules relative to this structure sheaf.

In the site $\Aff_{flat}/\Spec(R)$ faithfully flat descent of modules 
\cite[Remark I.2.19]{Mi} tells us that quasi-coherent modules are not only
 locally presentable, but globally presentable. We include the argument
 for completeness.
\begin{lemma} 
\label{modules}
The category quasi-coherent sheaves on $\Aff_{flat}/\Spec(R)$ 
is equivalent to the opposite category of $R$-modules.
\end{lemma}
\begin{proof}
Since $\otimes$ is right exact there is a functor from $R$-modules to 
quasi-coherent sheaves sending $M \to F_M$ where 
\[ F_M( \Spec R') = M\otimes_R R'. \]
It is clear that this functor is full and faithful.
Given a quasi-coherent sheaf $\cF$ on $\Aff_{flat}/\Spec(R)$
the definition of quasi-coherent implies that there is a 
cover $\{\Spec S_i \to \Spec R\}$ and $S_i$-modules $M_i$ so that
$$\cF|_{\Aff/\spec S_i}\cong M_i \otimes_{S_i} (-)$$
(since $\otimes$ is right exact). Evaluating $\cF$ on $S_i \otimes_R S_j$ we see that
\begin{equation} \label{1} 
M_i \otimes_R S_j \cong M_j \otimes_R S_i.
\end{equation}
$\cF(\spec R)$ is the equalizer
\begin{equation} \label{2}
\prod M_i \dbra \prod_{i,j} M_i \otimes_R S_j.
\end{equation}
Let $M$ be the value of this equalizer. Since $\spec S_i \ra \spec R$ is a cover
\begin{equation} \label{3}
R \ra \prod_i S_i \dbra \prod_{i,j} S_i \otimes_R S_j
\end{equation}
is exact on the left and remains so when we tensor with any $R$-module 
\cite[Proposition I.2.7, Remark I.2.19]{Mi}.
It follows that we can tensor equation \eqref{3} with $M_k$
and tensor \eqref{2} with $S_k$ to obtain the 
following isomorphisms 
$$\xymatrix{
\prod_i M_k \otimes_R S_i  \ar@2[r] \ar[d]^\cong & 
\prod_{i,j} M_k \otimes_R S_i \otimes_R S_j \ar[d]^\cong \\ 
\prod_i M_i \otimes_R S_k  \ar@2[r] & 
\prod_{i,j} M_i \otimes_R S_j \otimes_R S_k }$$ 
which induce an isomorphism between the equalizers $M \otimes_R S_k \cong M_k$.
Given a map of rings $R \to R'$, a similar argument shows that 
$\cF(\spec R')\cong M\otimes_R R'$, which completes the proof.
\end{proof}

\begin{prop}
\label{comodulesqc}
The category of quasi-coherent sheaves on $(\spec A,\Spec \Gamma)$ 
(or $\m_{(\spec A,\Spec \Gamma)}$) is
equivalent to the category of comodules on the Hopf algebroid $(A, \Gamma)$.
\end{prop}

\begin{proof}
An Hopf algebroid $(A,\Gamma)$ yields a diagram
$$\xymatrix{ A \ar@2[r]^L_R & \Gamma 
\ar@3[r]  & \Gamma\otimes_A \Gamma}$$
where the maps $\Gamma \to \Gamma \otimes_A \Gamma$ are $L\otimes {1_\Gamma}, \mu$ and
$1_\Gamma\otimes R$ with $\mu$ the comultiplication.
By Corollary \ref{descentsheaves} and Lemma \ref{modules}, a 
 quasi-coherent sheaf consists of an $A$ module $M$ and an isomorphism of 
 $\Gamma$-modules
$\Gamma \otimes_A M \llra{\alpha} M \otimes \Gamma$ making the following 
diagram commute:
$$\xymatrix{ \Gamma \otimes_A M \ar[r]^{\mu \otimes 1} \ar[dr]^{\alpha} & 
\Gamma \otimes_A \Gamma \otimes_A M \ar[r]^{1 \otimes \alpha}
\ar[dr]^{\alpha \otimes \mu} & \Gamma \otimes_A M \otimes_A \Gamma 
\ar[d]^{\alpha \otimes 1} \\ & M \otimes_A \Gamma \ar[r]^{1 \otimes \mu} 
& M \otimes_A \Gamma \otimes_A \Gamma.}$$
Let $\phi$ be defined as the composition
$$M \llra{R \otimes 1} \Gamma \otimes_A M \llra{\alpha} M \otimes_A \Gamma.$$
Precomposing the commutative diagram above with the map 
$M \llra{R \otimes 1} \Gamma \otimes_A M$ 
and using the identity $1\otimes R \otimes 1 = \mu\otimes 1 \circ R \otimes 1$, one
can see that the composition along the top and down to $M\otimes_A \Gamma \otimes_A \Gamma$
is $(\phi\otimes 1) \circ \phi$. The composition along the bottom is $(1 \otimes \mu) \circ \phi$.
so $\phi$ defines a comodule structure on $M$ (see \cite[Appendix A.1]{Ra}). 

Conversely a comodule structure on $M$ is a map of $A$ bimodules
$M \ra M \otimes_A \Gamma$ and so there is an extension of this map over
$M \llra{R \otimes 1} \Gamma \otimes_A M$ providing a $\Gamma$-module 
isomorphism $\Gamma \otimes_A M \llra{\alpha} M \otimes_A \Gamma$.
Another diagram chase shows that the comodule identity is equivalent to the
condition that $\alpha$ satisfies descent.
\end{proof}

The previous result together with Corollary \ref{equivqc}
yields the following Theorem of M. Hovey \cite[Theorems A and C]{Ho}.

\begin{cor}
\label{eq-ha}
Let $(A,\Gamma_A)$ and $(B,\Gamma_B)$ be two Hopf algebroids, for which 
$(\spec A,\spec \Gamma_A)$ and $(\spec B, \spec \Gamma_B)$ are
weakly equivalent in $P(\Aff_{flat},\Gr)_L$.
The category of $(A,\Gamma_A)$ comodules is equivalent to the category 
of $(B,\Gamma_B)$-comodules. 
\end{cor}

For the sake of completeness, we use the equivalence of categories of
Proposition \ref{comodulesqc}
to provide the reader with
an example of a category of quasicoherent sheaves which is not an
abelian category and conclude with some related remarks.
\begin{example}
Consider the Hopf algebroid $(\Z, \Z[\epsilon]/(p\epsilon,\epsilon^2))$
where the composition 
$$\Z[\epsilon]/(p\epsilon,\epsilon^2) \ra 
\Z[\epsilon_1]/(p\epsilon_1,\epsilon_1^2) \otimes_\Z
\Z[\epsilon_2]/(p\epsilon_2,\epsilon_2^2)$$
is given by sending $\epsilon$ to $(\epsilon_1 + \epsilon_2 +
\epsilon_1\epsilon_2)$.
Every abelian group $M$ has (at least) two comodule structures
$$M\ra \Z[\epsilon]/(p\epsilon,\epsilon^2) \otimes M$$
one of which is given by $1\otimes id_M$ and the other 
by $(1+\epsilon)\otimes id_M$. We denote these by $(M,1)$
and $(M,1+\epsilon)$ respectively. 
Consider the epimorphism $(\Z/p^2,1) \llra{r} (\Z/p,1)$.
There are monomorphisms
$$(\Z/p,1) \llra{i} (\Z/p^2,1), \quad \text{and} \quad
(\Z/p,1+\epsilon) \llra{i'} (\Z/p^2,1)$$
such that $r$ is the cokernel of both $i$ and $i'$.
Clearly $(\Z/p,1)$ and $(\Z/p,1+\epsilon)$
are not isomorphic but in an abelian category a monomorphism
must be the kernel of its cokernel.
\end{example}

Even when the category of quasi-coherent $\cO_{\m}$-modules is an abelian
category, the inclusion into $\cO_{\m}$-modules is not necessarily exact.
An example to consider is the multiplication by $p$ map
$$\cO_{\spec \Z} \llra{p} \cO_{\spec \Z}.$$
Let $K_p$ be the kernel of this map as an $\cO_{\spec \Z}$ module.
Then $K_p(\spec R)$ is the $p$-torsion in $R$ and so $K_p$ is not quasi-coherent.
The kernel of multiplication by $p$ within quasi-coherent modules
exists and is $0$. 

Furthermore, given a map $\spec R' \llra{f} \spec R$, pullback of sheaves is an exact functor
 but pullback of quasi-coherent sheaves is not in general: the pullback of a quasi-coherent sheaf
 $F_M$ on $\spec R$ is  
\[ f^*F_M(\spec R')=F_M(\spec R' \ra \spec R)  \cong F_{M\otimes_R R'}\]
 and therefore, for quasi-coherent sheaves,  the pullback functor corresponds to the tensor product 
$(-)\otimes_R R'$ which is not always exact.

\subsection{A different approach}
An alternate approach to the descent statements in this section
and the homotopy invariance of the previous section would be
to make use of the {\it stack of sheaves}, which we learned 
about from M. Hopkins.  

Disregarding set theoretic questions one can define a 
stack of sheaves $\cS \in P(\cC,\Gr)$ associating to $X \in \cC$ the 
groupoid of sheaves on $X$.  In a similar fashion one can define a 
category object in $P(\cC,\Gr)$, $(\cS,\cS_{map})$ where $\cS_{map}$
classifies maps between sheaves.

Given $\m \in P(\cC,\Gr)$ one could then define $Sh(\m)=\Hom(\m, \cS)$.  
With this definition, a sheaf on $\m$ would consist of a compatible assignment
of sheaves to each $X \in \cC$ and map $X \ra \m$.  
The equivalence of this definition with our definition of sheaf on $\cC/\m$
should come down to the equivalence of $\cC/\m$ with 
the homotopy colimit over $\cC/\m$ of the categories $\cC/X$.

A levelwise weak equivalence $\m \ra \m'$ induces an equivalence 
of sites $\cC/\m \we \cC/\m'$ and hence of categories of sheaves.
Since cofibrant replacement is a levelwise weak equivalence we 
would have 
$$\Hom(\m, \cS) \we h\Hom(\m, \cS).$$ 
An immediate corollary of this would be our homotopy invariance and 
descent results.

\end{document}